\documentclass{amsart}
\usepackage{amsmath}
\usepackage{graphicx}
\usepackage{latexsym}
\newtheorem{thm}{Theorem}[section]

\newtheorem{cor}[thm]{Corollary}
\newtheorem{prop}[thm]{Proposition}
\newtheorem{rem}[thm]{Remark}
\newtheorem{conj}[thm]{Conjecture}

\newtheorem{defn}[thm]{Definition}

\begin{document}

\title[Lefschetz fibrations with small slope]
{Lefschetz fibrations with small slope}

\author[N. Monden]{Naoyuki Monden}
\address{Department of Mathematics, Graduate School of Science, Kyoto University, Kyoto, Japan}
\email{n-monden@math.kyoto-u.ac.jp}

\begin{abstract}
We construct Lefschetz fibrations over $S^2$ which do not satisfy the slope inequality. 
This gives a negative answer to a question of Hain.  \\[-20pt]
\end{abstract}

\maketitle

\setcounter{secnumdepth}{2}
\setcounter{section}{0}


\section{Introduction}
By the remarkable works of Donaldson \cite{Do} and Gompf \cite{GS}, it turned out that Lefschetz fibrations 
are closely connected with symplectic 4-manifolds. 
As a result, the study of Lefschetz fibrations has been an active area of research. 
In this paper, we consider the geography problem of Lefschetz fibrations over $S^2$ which derives from that of complex surfaces fibred over curves. 

We introduce two kinds of 
geography problems. 
Let $\sigma$ and $e$ be the signature and the Euler characteristic of 
a closed oriented smooth $4$-manifold $X$, respectively. 
For an almost complex closed $4$-manifold $X$, we set $K^2:=3\sigma+2e$ 
and $\chi_h:=(\sigma+e)/4$ (the \textit{holomorphic Euler characteristic}).

One is the geography problem for complex surfaces 
(i.e. the characterization of pairs $(K^2,\chi_h)$ corresponding to minimal complex surfaces. 
It is well-known that for a minimal complex surface of general type, $K^2>0$, $\chi_h>0$ 
and $2\chi_h-6\leq K^2\leq 9\chi_h$. 
The latter two inequalities are called the \textit{Noether-} and 
\textit{Bogomolov-Miyaoka-Yau-inequalities} (cf. \cite{BPV}). 
The above geography problem can be extend to the symplectic 4-manifolds. 
However, there exists minimal symplectic manifolds which do not satisfy the Noether inequality. 
Fintushel and Stern \cite{FS1} constructed Lefschetz fibration which does not satisfy the Noether inequality. 
In particular, for most pairs $(p,q)$ satisfying $p<2q-6$, there exists a minimal symplectic 4-manifold with $p=K^2$ and $q=\chi_h$ (cf. \cite{GS}).
On the other hand, no examples of a minimal symplectic 4-manifold with $K^2>9\chi_h$ have been found yet.

The other is the geography problem for complex surfaces fibred over curves. 
Hereafter, we assume $g\geq 2$. 
Let $f:S\rightarrow C$ be a relatively minimal holomorphic genus-$g$ fibration, 
where $S$ is a complex surface and $C$ is a complex curve of genus $k$. 
We define relative numerical invariants $\chi_f:=\chi_h-(g-1)(k-1)$ and 
$K_f^2:=K^2-8(g-1)(k-1)$ for $f:S\rightarrow C$. 
Then, we have two inequalities $\chi_f\geq 0$ and $K_f^2\geq 0$ known as 
\textit{Beauville's inequality} (cf. \cite{Be}) and \textit{Arakelov's inequality} (cf. \cite{Ar}), respectively. 
For $\chi_f\neq 0$, which is equivalent to the fact that $f$ is not a holomorphic bundle, we define $\lambda_f$ to be the quotient $K_f^2/\chi_f$. 
We call $\lambda_f$ the slope of $f$. 
Xiao \cite{Xi} proved that $4-4/g\leq \lambda_f\leq 12$ 
(i.e., $(4-4/g)\chi_f\leq K_f^2\leq 12\chi_f$). 
The former inequality is called the \textit{slope inequality}. 

The study of the slope of holomorphic fibrations was mainly motivated by Severi's inequality
, which states that if $S$ is a minimal surface of general type of maximal Albanese dimension, then $K^2\geq 4\chi_h$. 
In another words, if $K^2< 4\chi_h$, then $S$ is a surface fibred over $C$ of genus $b_1(S)/2$. 
Severi \cite{Se} claimed it in 1932, but his proof was not correct (cf. \cite{Cat}). 
The inequality was independently posed as a conjecture by Reid \cite{Re} and by Catanese \cite{Cat}. 
Xiao \cite{Xi} proved the conjecture when $S$ is a surface fibred over a curve of positive genus. 
He showed that if $S$ admits a holomorphic genus $g$ fibration $f$ over $C$ of positive genus $k$ with 
$K^2<4\chi_h+4(g-1)(k-1)$ (i.e. $\lambda_f<4$), then $k=b_1(S)/2$. 
Konno \cite{Kon96} proved the inequality in the case of \textit{even} surfaces. 
The conjecture was solved by Manetti \cite{Man03} when $S$ has ample canonical bundle. 
Pardini \cite{Pa} proved the conjecture completely by using the slope inequality for holomorphic fibrations over $\mathbb{CP}^1$.


Let $f:X\rightarrow S^2$ be a relatively minimal genus-$g$ Lefschetz fibration with $n$ singular fibers. 
Then, $\chi_f$, $K^2_f$ and the slope $\lambda_f$ are defined in the same way as for complex surfaces fibred over curves. 
From $e(X)=-4(g-1)+n$ and the results of Ozbagci \cite{Oz} and Stipsicz \cite{St}, we have $\chi_f\geq 0$, $K_f^2\geq 4g-4$ and $\lambda_f\leq 10$. 
By the result of Li \cite{Li}, we find that $\chi_f=0$ if and only if $n=0$ (i.e., $X=\Sigma_g\times S^2$). 
Moreover, it is well-known that 
any hyperelliptic Lefschetz fibrations 
satisfy the slope inequality. 
Therefore, genus-$2$ Lefschetz fibrations satisfy the slope inequality. 
In particular, if $f$ is a hyperelliptic Lefschetz fibration with only nonseparating vanishing cycles, 
then $\lambda_f$ is equal to $4-4/g$. 
To author's knowledge, the slope of all known Lefschetz fibrations over 
$S^2$ is greater than or equal to $4-4/g$.
\begin{conj}[Hain cf. \cite{ABKP}, Question 5.10, \cite{EN}, Conjecture 4.12]\label{slope} 
For every genus-$g$ Lefschetz fibration $f: X\rightarrow S^2$, 
the slope inequality $\lambda_f \geq 4-4/g$ holds.
\end{conj}

We can reformulate the slope inequality of Conjecture~\ref{slope} in terms of the Deligne-Mumford compactified moduli space of stable curves of genus $g$, denoted by $\overline{\mathcal{M}_g}$, as follows. 
For a relatively minimal genus-$g$ Lefschetz fibration $f:X\rightarrow S^2$ with $n$ singular fibers, 
we can obtain a symplectic structure on $X$ such that for all $x\in S^2$, $f^{-1}(x)$ is a pseudo-holomorphic curve.
Since a 2-dimensional almost-complex structure is integrable, $f^{-1}(x)$ determines a point in $\overline{{\mathcal M}_g}$. 
Thus, we can obtain the moduli map $\phi_f:S^2\rightarrow \overline{{\mathcal M}_g}$ which is defined by $\phi_{f}(x)=[f^{-1}(x)]\in\overline{\mathcal{M}_g}$ for $x\in S^{2}$. 
We denote by ${\mathcal H}_g$ the Hodge bundle on $\overline{\mathcal{M}_g}$ with fiber the determinant line $\wedge^{g}H^{0}(C;K_{C})$, where $C$ is the set of critical points of $f$. 
By using Smith's signature formula \cite{Sm2}, we have the following inequality which is equivalent to the slope inequality of Conjecture~\ref{slope}. 
\begin{eqnarray*}
(8g+4)\langle c_1({\mathcal H}_g),[\phi_f(S^2)] \rangle - g\cdot n \geq 0
\end{eqnarray*}


We give a negative answer to Conjecture~\ref{slope}.
\begin{thm}\label{nonholo}
For $g\geq 3$, there exist a genus-$g$ Lefschetz fibration over $S^2$ with slope $\lambda_f=4-1/g-1/3g$ 
whose total space is simply connected and non-spin. 
\end{thm}
Moreover, by fiber sum operations, we have the following results: 
\begin{cor}\label{nonholo2}
For each $g\geq 3$, $m\geq 0$ and $l\geq 0$, there a genus-$g$ Lefschetz fibration $f_{m,l} : X_{m,l}\rightarrow S^2$ with 
slope $\lambda_{f_{m,l}}=4-4/g-1/(m+3)g$ s.t. $\pi_1(X_{m,l})=1$. 
Moreover, if $(m,l)\neq (0,0)$ (resp. $l\not\equiv 0 \ {\rm mod} \ 16$), then $X_{m,l}$ is minimal (resp. non-spin) symplectic 4-manifold. 
\end{cor}
\begin{cor}\label{nonholo3}
For each $g\geq 3$, $m\geq 1$ and $l\geq 0$, there exist a genus-$g$ Lefschetz fibration $f^\prime_{m,l} : Y_{m,l}\rightarrow S^2$ with 
slope $\lambda_{f^\prime_{m,l}}=4-4/g-1/2g+1/(2\cdot 3^{m-1}g)$ s.t. $\pi_1(Y_{m,l})=1$. 
Moreover, If $l\geq 1$ (resp. 
$m$ is even and $l\not\equiv 0 \ {\rm mod} \ 16$), then $Y_{m,l}$ is minimal (resp. non-spin) symplectic 4-manifold. 
\end{cor}
As a consequence, the Lefschetz fibrations in Theorem~\ref{nonholo}, Collorary~\ref{nonholo2} and~\ref{nonholo3} are non-holomorphic (Corollary~\ref{nonholomorphic}).

We have the following natural question: what Lefschetz fibrations satisfy the slope inequality ? 
By combining the results of \cite{St}, \cite{St2} and \cite{Li}, we can show that Lefschetz fibrations with $b_2^+=1$ satisfy the slope inequality. 
Stipsictz showed that if $X\rightarrow S^2$ is a relatively minimal genus-$g$ Lefschetz fibration over $S^2$ with $b_+^2(X)=1$ 
and $X$ is not diffeomorphic to the blow-up of a ruled surface (i.e., diffeomorphic to a $S^2$-bundle over $\Sigma_k$), 
then $b_1(X)\in\{0,2\}$ and $e\geq 0$ (see \cite{St}, Corollary 3.3 and 3.5). 
In particular, if $X$ is the blow-up of a $S^2$-bundle over $\Sigma_k$, then $k\leq g/2$ (see \cite{Li}, Proposition 4.4). 
Then, we obtain the following result. 
\begin{thm}\label{lower}
Let $f:X\rightarrow S^2$ be a genus-$g$ Lefschetz fibration with $b_2^+(X)=1$ for $g\geq 2$. 
Suppose that $X$ is not diffeomorphic to the blow-up of a ruled surface. 

{\rm (1)} If $b_1(X)=0$, then $4-4/g\leq \lambda_f\leq 8+1/g$. 

{\rm (2)} If $b_1(X)=2$, then $4-4/g\leq \lambda_f\leq 8$. \\[0.5mm]
Suppose that $X$ is the blow-up of a $S^2$-bundle over $\Sigma_k$ $(0\leq k\leq g/2)$. 

{\rm (3)} $(4-4/g\leq) \ 4+4(k-1)/(g-k)\leq \lambda_f\leq 8$. This lower bound is sharp. 

\end{thm}

\vspace{0.1in}
\noindent \textit{Acknowledgments.} The author would like to thank 
Hisaaki Endo for pointing out 
that the author's idea is effective in constructing a Lefschetz fibration which violates the slope inequality 
and for his encouragement. 
The author would like to express his gratitude to 
Kazuhiro Konno for his kind explanation of the history of the slope, and Mustafa Korkmaz and Kouich Yasui for their comments on this paper.

\section{Preliminaries} \
Let $\Sigma_g$ be a closed oriented surface of genus $g\geq 2$ 
and let $\Gamma_g$ be the mapping class group of $\Sigma_g$.
We denote by $t_c$ the right handed Dehn twist about a simple closed curve $c$ 
on an oriented surface. 
$t_ct_d$ means that we first apply $t_d$ then $t_c$.

We begin by recalling the definition and basic properties of Lefschetz fibrations.
(More details can be found in \cite{GS}.) 
\begin{defn}\label{LF}\rm
Let $X$ be a closed, oriented smooth $4$-manifold. 
A smooth map $f : X \rightarrow S^2$ is a genus-$g$ Lefschetz fibration 
if it satisfies the following condition : \\
(i) $f$ has finitely many critical values $b_1,\ldots,b_n \in S^2$, and 
$f$ is a smooth $\Sigma_g$-bundle over $S^2-\{b_1,\ldots,b_n\}$, \\
(ii) for each $i$ $(i=1,\ldots,n)$, there exists a unique critical point $p_i$ 
in the \textit{singular fiber} $f^{-1}(b_i)$ such that about each $b_{i}$ and 
$f^{-1}(b_{i})$ there are complex local coordinate charts agreeing with 
the orientations of $X$ and $S^2$ on which $f$ is of the form 
$f(z_{1},z_{2})=z_{1}^{2}+z_{2}^{2}$, \\
(i\hspace{-.1em}i\hspace{-.1em}i) $f$ is relatively minimal (i.e. no fiber contains a $(-1)$-sphere.)

A Lefschetz fibration $f : X \rightarrow S^2$ is \textit{holomorphic} if there are complex structures
on both $X$ and $S^2$ with holomorphic projection $f$.
\end{defn}
Each singular fiber is obtained by collapsing a simple closed curve (the \textit{vanishing cycle}) 
in the regular fiber. 
The monodromy of the fibration around a singular fiber is given by 
a right handed Dehn twist along the corresponding vanishing cycle.

Once we fix an identification of $\Sigma_{g}$ with the fiber over a base point of $S^2$, 
we can characterize the Lefschetz fibration  $f:X\rightarrow S^2$ by its \textit{monodromy representation} 
$\pi_{1}(S^2-\{b_1,\ldots,b_n\})\rightarrow \Gamma_{g}$. 
Let $\gamma_1,\ldots,\gamma_n$ be an ordered system of generating loops for $\pi_{1}(S^2-\{b_1,\ldots,b_n\})$, 
such that each $\gamma_i$ encircles only $b_i$ and $\prod \gamma_i$ is homotopically trivial. 
Thus, the monodromy of $f$ 
comprises a factorization 
\begin{eqnarray*}
t_{v_1}t_{v_2}\cdots t_{v_n}=1\in \Gamma_g, 
\end{eqnarray*}
where $v_i$ are 
vanishing cycles of the singular fibers. 
This factorization is called the \textit{positive relator}.

According to theorems of Kas \cite{Kas} and Matsumoto \cite{Ma}, 
if $g\geq 2$, then 
the isomorphism class of a Lefschetz fibration is determined by a positive relator 
modulo simultaneous conjugations 
\begin{eqnarray*}
t_{v_1}t_{v_2}\cdots t_{v_n} \sim t_{\phi(v_1)}t_{\phi(v_2)}\cdots t_{\phi(v_n)} \ \ {\rm for \ all} \ \phi \in \Gamma_g
\end{eqnarray*}
and elementary transformations 
\begin{eqnarray*}
t_{v_1}\cdots t_{v_i}t_{v_{i+1}}t_{v_{i+2}}\cdots t_{v_n} \sim t_{v_1}\cdots t_{v_{i+1}}t_{t_{v_{i+1}}^{-1}(v_i)}t_{v_{i+2}}\cdots t_{v_n},\\
t_{v_1}\cdots t_{v_i}t_{v_{i+1}}t_{v_{i+2}}\cdots t_{v_n} \sim t_{v_1}\cdots,t_{t_{v_i}(v_{i+1})}t_{v_i}t_{v_{i+2}}\cdots t_{v_n}.
\end{eqnarray*}
Note that $\phi t_{v_i}\phi^{-1}=t_{\phi(v_i)}$ and $t_{v_{i+1}}^{-1}t_{v_i}t_{v_{i+1}}=t_{t_{v_{i+1}}^{-1}(v_i)}$. 
For all $\phi\in \Gamma_g$, let $\phi(\varrho)$ be the positive relator which is obtained by applying simultaneous conjugations by $\phi$ to a positive relator $\varrho$. 
We denote a Lefschetz fibration associated to a positive relator $\varrho \ \in \Gamma_g$ by $f_\varrho:X_\varrho \rightarrow S^2$. 
Clearly, if $\varrho_1\sim\varrho_2$ in $\Gamma_g$ (i.e. $\varrho_2$ is obtained by applying elementary transformations or simultaneous conjugations to $\varrho_1$)
, then $\chi_{f_{\varrho_1}}=\chi_{f_{\varrho_2}}$ and $K_{f_{\varrho_1}}^2=K_{f_{\varrho_2}}^2$.

For positive relators $\varrho_1$ and $\varrho_2$ in $\Gamma_g$, the genus-$g$ Lefschetz fibration $f_{\varrho_1\varrho_2}:X_{\varrho_1\varrho_2}\rightarrow S^2$ is 
the (trivial) fiber sum of $f_{\varrho_1}$ and $f_{\varrho_2}$. 
Since $\sigma(X_{\varrho_1\varrho_2})=\sigma(X_{\varrho_1})+\sigma(X_{\varrho_2})$ and $e(X_{\varrho_1\varrho_2})=e(X_{\varrho_1})+e(X_{\varrho_2})+4(g-1)$, 
we see $\chi_{f_{\varrho_1\varrho_2}}=\chi_{f_{\varrho_1}}+\chi_{f_{\varrho_2}}$ and $K_{f_{\varrho_1\varrho_2}}^2=K_{f_{\varrho_1}}^2+K_{f_{\varrho_2}}^2$. 
In particular, if $\varrho_1\sim\varrho_2$, then $\chi_{f_{\varrho_1\varrho_2}}=2\chi_{f_{\varrho_1}}(=2\chi_{f_{\varrho_2}})$ and $K_{f_{\varrho_1\varrho_2}}^2=2K_{f_{\varrho_1}}^2(=2K_{f_{\varrho_2}}^2)$.

We next begin with a definition of the lantern relation \ (see \cite{De}, \cite{Jo}).\\
\begin{figure}[h]
 \begin{center}
      \includegraphics[width=3cm]{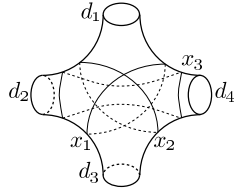}
      \caption{The curves $d1,d_2.d_3,d_4,x_1,x_2,x_3$.}
      \label{LR}
   \end{center}
 \end{figure}
\begin{defn}\label{lantern}\rm
Let $\Sigma_0^{4}$ denote a sphere with $4$ boundary components. 
Let $d$, $d_{1}$, $d_{2}$, $d_{3}$, $d_4$ be the $4$ boundary curves of $\Sigma_0^{4}$ 
and let $x_1$, $x_2$, $x_3$ be the interior curves as shown in Figure~\ref{LR}. 
Then, we have the \textit{lantern relation}
\begin{eqnarray*}
t_{d_1}t_{d_2}t_{d_3}t_{d_4}=t_{x_1}t_{x_2}t_{x_3}.
\end{eqnarray*}
\end{defn}
Let $\varrho$ be a positive relator of $\Gamma_g$. 
Let $d_1$, $d_2$, $d_3$, $d_4$, $x_1$, $x_2$, $x_3$ be curves as in Definition~\ref{lantern}. 
Suppose that $\varrho$ includes $t_{d_1}t_{d_2}t_{d_3}t_{d_4}$ as a subword : 
\begin{eqnarray*}
\varrho=U\cdot t_{d_1}t_{d_2}t_{d_3}t_{d_4} \cdot V, 
\end{eqnarray*}
where $U$ and $V$ are products of right handed Dehn twists. 
Then, by the lantern relation, 
the product of right handed Dehn twists 
\begin{eqnarray*}
\varrho^{\prime}=U\cdot t_{x_1}t_{x_2}t_{x_3} \cdot V 
\end{eqnarray*}
is also a positive relator of $\Gamma_g$.

This operation is one of substitution techniques introduced by Fuller. 
\begin{defn}\rm
We say that $\varrho^{\prime}$ is obtained by applying an $L$-\textit{substitution} to $\varrho$. 
Conversely, $\varrho$ is said to be obtained by applying an $L^{-1}$-\textit{substitution} to $\varrho^{\prime}$. 
We also call these two kinds of operations \textit{lantern substitutions}. 
\end{defn}
\begin{prop}[Endo and Nagami, \cite{EN}, Theorem 4.3 and Proposition 3.12]\label{EN}
Let $\varrho$, $\varrho^{\prime}$ be positive relators of $\Gamma_g$ 
and let $X_\varrho$, $X_{\varrho^{\prime}}$ be the corresponding 
Lefschetz fibrations over $S^2$, respectively. 
Suppose that $\varrho$ is obtained by applying an $L^-$-substitution 
to $\varrho^\prime$. 
Then, $\sigma(X_\varrho)=\sigma(X_{\varrho^{\prime}})-1$ and $e(X_\varrho)=e(X_{\varrho^{\prime}})+1$. 
Therefore, 
\begin{eqnarray*}
\chi_{f_\varrho}=\chi_{f_{\varrho^\prime}},\ \ \ K_{f_\varrho}^2=K_{f_{\varrho^\prime}}^2-1.
\end{eqnarray*}
\end{prop}
\begin{rem}\rm
Endo and Gurtas \cite{EG} showed that $X_{\varrho^{\prime}}$ is a rational blowdown of $X_\varrho$ introduced by Fintushel and Stern \cite{FS0}. 
Such relations were also generalized by Endo, Mark, and Van Horn-Morris \cite{EMHM}. 
\end{rem}

\section{Proofs of Main results}
In order to prove Theorem~\ref{nonholo}, we will need the following positive relator: 
\begin{figure}[hbt]
 \begin{center}
      \includegraphics[width=9.5cm]{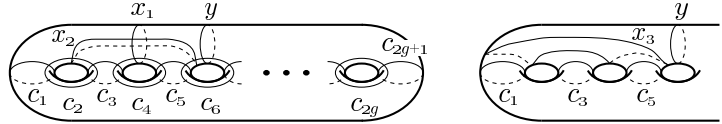}
      \caption{The curves $c_1,\ldots,c_{2g+1}$, $x_1,x_2,x_3,y$.}
      \label{chain}
   \end{center}
 \end{figure}
 
Suppose $g\geq 3$. 
Let $c_1,c_2,\ldots, c_{2g+1}$ be the curves in $\Sigma_g$ as shown in Figure~\ref{chain}. 
We denote by $h_g$ $(\in \Gamma_g)$ the product of $8g+4$ right handed Dehn twists 
\begin{eqnarray*}
h_g \ : \ = \ (t_{c_1}t_{c_2}\cdots t_{c_{2g+1}}^2\cdots t_{c_2}t_{c_1})^2. 
\end{eqnarray*}
It is well known that $h_g$ is a positive relator in $\Gamma_g$ and that 
$\sigma(X_{h_g})=-4(g+1)$ and $e(X_{h_g})=4(g+2)$. 
This gives $\chi_{f_{h_g}}=g$, $K_{f_{h_g}}^2=4g-4$ and $\lambda_{f_{h_g}}=4-4/g$ (i.e. $f_{h_g}$ is lying on the slope line). 
By Lemma 3.2 (c) of \cite{ABKP}, we have $\pi_1(X_{h_g})=1$. 
\begin{proof}[Proof of Theorem~\ref{nonholo}]
Let $x_1,x_2,x_3, y$ be the curves as shown in Figure~\ref{chain}. 
Since $c_1, x_i$ are nonseparating curves, there exists a diffeomorphism $f_i$ such that $\phi_i(c_1)=x_i$. 
Hence, we have the following positive relator $r_i$ $(i=1,2,3)$: 
\begin{eqnarray*}
r_i=\phi_ih_g\phi_i^{-1}&=&\phi_i(t_{c_1}t_{c_2}\cdots t_{c_{2g+1}}^2\cdots t_{c_2}t_{c_1})^2\phi_i^{-1}\\
&=&(t_{\phi_i(c_1)}t_{\phi_i(c_2)}\cdots t_{\phi_i(c_{2g+1})}^2\cdots t_{\phi_i(c_2)}t_{\phi_i(c_1)})^2\\
&=&(t_{x_i}t_{\phi_i(c_2)}\cdots t_{\phi_i(c_{2g+1})}^2\cdots t_{\phi_i(c_2)}t_{\phi_i(c_1)})^2=1 \ \in \ \Gamma_g .
\end{eqnarray*}
Let $r^{\prime}_g=r_1r_2r_3$. 
Since $f_{r^{\prime}_g}$ is the fiber sum of $f_{r_1}$, $f_{r_2}$ and $f_{r_3}$ which are obtained by applying simultaneous conjugations to $h_g$, we have 
\begin{eqnarray*}
\chi_{f_{r^\prime_g}}=3\chi_{f_{h_g}}=3g, \ \ \ K_{f_{r^\prime_g}}^2=3K_{f_{h_g}}^2=3(4g-4). 
\end{eqnarray*}
We apply elementary transformations to $r^{\prime}_g$ as follows:
\begin{eqnarray*}
&r^{\prime}_g& = r_1r_2r_3\\
 &=& t_{x_1}t_{\phi_1(c_2)}\cdots t_{\phi_1(c_2)}\underline{t_{\phi_1(c_1)}\cdot t_{x_2}}t_{\phi_2(c_2)}\cdots t_{\phi_2(c_1)}\cdot t_{x_3}t_{\phi_3(c_2)}\cdots t_{\phi_3(c_1)}\\
 &\sim& t_{x_1}t_{\phi_1(c_2)}\cdots t_{\phi_1(c_2)}\underline{t_{x_2}t_{t_{x_2}^{-1}(\phi_1(c_1))}}t_{\phi_2(c_2)}\cdots t_{\phi_2(c_1)}\cdot t_{x_3}t_{\phi_3(c_2)}\cdots t_{\phi_3(c_1)}\\
 &\rotatebox{90}{$\cdots$}& \hspace{30mm} \rotatebox{90}{$\cdots$}\\
 &\sim& t_{x_1}t_{x_2}t_{t_{x_2}^{-1}(\phi_1(c_2))}\cdots t_{t_{x_2}^{-1}(\phi_1(c_2))}t_{t_{x_2}^{-1}(\phi_1(c_1))}t_{\phi_2(c_2)}\cdots \underline{t_{\phi_2(c_1)}\cdot t_{x_3}}t_{\phi_3(c_2)}\cdots t_{\phi_3(c_1)}\\
 &\sim& t_{x_1}t_{x_2}t_{t_{x_2}^{-1}(\phi_1(c_2))}\cdots t_{t_{x_2}^{-1}(\phi_1(c_2))}t_{t_{x_2}^{-1}(\phi_1(c_1))}t_{\phi_2(c_2)}\cdots \underline{t_{x_3}t_{t_{x_3}^{-1}(\phi_2(c_1))}}t_{\phi_3(c_2)}\cdots t_{\phi_3(c_1)}\\
 &\rotatebox{90}{$\cdots$}& \hspace{30mm} \rotatebox{90}{$\cdots$}\\
 &\sim& (t_{x_1}t_{x_2}t_{x_3})W,
\end{eqnarray*}
where $W$ is a product of $24g+9$ right handed Dehn twists.
By the lantern relation, we get the following positive relator $r_g$ :
\begin{eqnarray*}
r_g \ : \  = \ (t_{c_1}t_{c_3}t_{c_5}t_y)W.
\end{eqnarray*}
Since $r_g$ is obtained by applying an $L^{-1}$-substitution to $r^{\prime}_g$, by Proposition~\ref{EN}
\begin{eqnarray*}
\chi_{f_{r_g}}=3g, \ \ \ K_{f_{r_g}}^2=3(4g-4)-1.
\end{eqnarray*}
Then, the slope of $f_{r_g}$ is equal to $4-4/g-1/3g$.

Since it is easy to check that $W$ includes the Dehn twist about a curve $\phi_3(c_i)$ for $1\leq i\leq 2g+1$, $\pi_1(X_{r_g})=\pi_1(X_{h_g})=1$ by Lemma 3.2 (c) of \cite{ABKP}. 
From Theorem of \cite{Ro} and $\sigma(X_{r_g})=3(-4(g+1))-1$, we see that $X_{r_g}$ is non-spin. 
This completes the proof of Theorem~\ref{nonholo}. 
\end{proof}
\begin{rem}\rm
Since $r_g$ is is obtained by applying an $L^{-}$-substitution to $r^\prime_g$, $X_{r_g}$ is a \textit{rational blowup} of $X_{r^\prime_g}$. 
By applying elementary transformations to a relator corresponding to a Lefschetz fibration which is obtained by taking a twisted fiber sum with sufficiently many Lefschetz fibrations, 
we obtain a positive relator such that we can apply a monodromy substitution, 
which corresponds to the operation of rational blowdown (resp. rational blowup) in \cite{EMHM}, to it. 
\end{rem}

\begin{rem}\rm
Miyachi and Shiga \cite{MS} produced genus-$g$ Lefschetz fibrations over $\Sigma_{2m}$ which do not satisfy the slope inequality.
\end{rem}

\begin{proof}[Proof of Corollary~\ref{nonholo2}]
For any $m\geq 0$, we consider the Lefschetz fibration $f_{r_gh_g^m}:X_{r_gh_g^m}\rightarrow S^2$ which is the fiber sum of $f_{r_g}$ and $m$ copies of $f_{h_g}$. Then, 
\begin{eqnarray*}
\chi_{f_{r_gh_g^m}}&=&\chi_{f_{r_g}}+m\chi_{f_{h_g}}=(3+m)g, \\
K_{f_{r_gh_g^m}}^2&=&K_{f_{r_g}}^2+mK_{f_{h_g}}^2=(3+m)(4g-4)-1.
\end{eqnarray*}
Therefore, we obtain $\lambda_{f_{r_gh_g^m}}=4-4/g-1/(m+3)g$. 

Let $f_{m,l} : X_{m,l}\rightarrow S^2$ be the fiber sum of $l$ copies of $f_{r_gh_g^m}$ (i.e. $f_{m,l}=f_{(r_gh_g^m)^l}$). 
By $\lambda_{f_{r_gh_g^m}}=4-4/g-1/(m+3)g$ and $\pi_1(X_{r_g})=1$, we have $\lambda_{f_{m,l}}=4-4/g-1/(m+3)g$ and $\pi_1(X_{m,l})=1$. 
By the result of Usher \cite{Us}, $X_{m,l}$ is minimal for $(m,l)\neq(0,0)$. 
From Rohlin's Theorem and $\sigma(X_{m,l})=l\{-4(g+1)(m+3)-1\}$, for $l\not\equiv 0 \ {\rm mod} \ 16)$, $X_{m,l}$ is non-spin. 
This completes the proof. 
\end{proof}

\begin{proof}[Proof of Corollary~\ref{nonholo3}]
Let $\varrho_1=h_g$ and $\varrho_2=r_g$
When we apply the argument of Theorem~\ref{nonholo} again, with $\varrho_1=h_g$ replaced by $\varrho_2=r_g$, we obtain a genus-$g$ Lefschetz fibration $f_{\varrho_3} : X_{\varrho_3}\rightarrow S^2$ with 
\begin{eqnarray*}
\chi_{f_{\varrho_3}}&=&3\chi_{f_{\varrho_2}}=3\cdot 3\chi_{f_{\varrho_1}}\\
K_{f_{\varrho_3}}^2&=&3K_{f_{\varrho_2}}^2-1=3(3K_{f_{\varrho_1}}-1)-1.
\end{eqnarray*}
By repeating this argument, we get a genus-$g$ Lefschetz fibration $f_{\varrho_m}$ $(m\geq 1)$ with 
\begin{eqnarray*}
\chi_{f_{\varrho_m}}&=&3^{m-1}\chi_{f_{\varrho_1}}=3^{m-1}g\\
K_{f_{\varrho_m}}^2&=&3(\cdots (3(3K_{f_{\varrho_1}}^2-1)-1)\cdots)-1=3^{m-1}K_{f_{\varrho_1}}-3^{m-2}-\cdots -3-1\\
&=&3^{m-1}(4g-4)-(3^{m-1}-1)/2.
\end{eqnarray*}
Therefore, $\lambda_{f_{\varrho_m}}=4-4/g-1/2g+1/(2\cdot 3^{m-1}g)$. 

Let $f^\prime_{m,l} : Y_{m,l}\rightarrow S^2$ be the fiber sum of $l$ copies of $f_{\varrho_m}$, and so $\lambda_{f^\prime_{m,l}}=4-4/g-1/2g+1/(2\cdot 3^{m-1}g)$. 
By an argument similar to the proof of Theorem~\ref{nonholo}, we see $\pi_1(Y_{m,l})=1$. 
By the result of Usher, $Y_{m,l}$ is minimal for $l\geq 1$. 
From Rohlin's Theorem and $\sigma(Y_{m,l})=l\{3^{m-1}(-4(g+1))-(3^{m-1}-1)/2\}$, $Y_{m,l}$ is non-spin if $m$ is even and $l\not\equiv 0 \ {\rm mod} \ 16$. 
This completes the proof.
\end{proof}

From the slope inequality for holomorphic fibrations, we have the following necessary condition for a Lefschetz fibration 
to be holomorphic :
\begin{prop}[Xiao, \cite{Xi}]\label{xi}
If a Lefschetz fibration $f$ is holomorphic, then the slope inequality 
$\lambda_f\geq 4-4/g$ holds.
\end{prop}
As a consequence, we have the following results.
\begin{cor}\label{nonholomorphic}
The Lefschetz fibrations of Theorem~\ref{nonholo}, Corollary~\ref{nonholo2} and~\ref{nonholo3} are non-holomorphic. 
\end{cor}

\begin{rem}\rm
There are various kinds of non-holomorphic Lefschetz fibrations. 
By fiber summing two copies of genus-$2$ Lefschetz fibration due to Matsumoto \cite{Ma}, 
Ozbagci and Stipsicz \cite{OS} constructed non-holomorphic genus-$2$ Lefschetz fibrations 
whose total space does not admit a complex structure. 
Korkmaz \cite{Kor} generalized their examples to $g\geq 3$. 
The above examples of Fintushel and Stern are also non-holomorphic Lefschetz fibrations. 
From study of divisors in moduli space, Smith \cite{Sm3} showed that a genus-$3$ Lefschetz fibration 
over $S^2$ which was produced by Fuller is non-holomorphic. 
Endo and Nagami \cite{EN} constructed some examples of non-holomorphic Lefschetz fibrations 
which violate lower bounds of the slope for non-hyperelliptic fibrations of genus 
$3$, $4$ and $5$ from the results of Konno \cite{Kon1}, \cite{Kon2} and Chen \cite{Ch}. 
Hirose \cite{Hi} also gave some examples of $g=3,4$. 
\end{rem}


\begin{proof}[Proof of Theorem~\ref{lower}]
Let $f:X\rightarrow S^2$ be a nontrivial genus-$g$ Lefschetz fibration with $b_2^+(X)=1$. 
Note that $-4(g-1)\leq K^2$, and so $4(g-1)\leq K_f^2$ (see \cite{St}, Lemma 3.2). 
Suppose that $X$ is not diffeomorphic to the blow-up of a ruled surface. 

First, suppose $b_1=0$. 
Since $b_2^+=1$ and $\chi_f=(\sigma+e)/4+(g-1)=(b_2^+-b_1+1)/2+(g-1)=g$, we have $4(g-1)/g\leq K_f^2/\chi_f=\lambda_f$. 
On the other hand, since $K^2=3\sigma+2e=5b_2^+-b_2^-+4-4b_1=9-b_2^-$, by $b_2^-\geq 0$, we have $\lambda_f= K_f^2/\chi_f=\{9-b_2+8(g-1)\}/g\leq 8+1/g$. 

Next, suppose $b_1=2$. 
Then, $\chi_f=g-1$. 
Therefore, by $4(g-1)\leq K_f^2$, we have $4\leq \lambda_f$. 
Since $0\leq e=2-2b_1+b_2^++b_2^-=2-4+1+b_2^-=-1+b_2^-$, we obtain $\lambda_f= \{1-b_2^-+8(g-1)\}/(g-1)\leq 8$. 

Finally, suppose that $X$ is the $m$-fold blow-up of a $S^2$-bundle over $\Sigma_k$. 
Let $Y$ be the $S^2$-bundle over $\Sigma_k$. 
Then, since $b_1(Y)=2k$, $b_2^\pm(Y)=1$ and $X= Y\sharp m\overline{\mathbb{CP}^2}$, we have $b_1(X)=2k$, $b_2^+(X)=1$, $b_2^-(X)=m+1$, $e(X)=4-4k+m$ and $\sigma(X)=-m$. 
Hence, we have $\lambda_f=8-m/(g-k)$. 
From $m\geq 0$, $\lambda_f\leq 8$. 
We will give lower bounds for $\lambda_f$. 
By Lemma 3.2 in \cite{St2}, $4(2k-g)+ m\leq 4$. 
This gives $2k\leq g$ or $2k=g+1$ and $m=0$. 
By Proposition 4.4 in \cite{Li}, we need only consider $2k\leq g$. 
From $\lambda_f=8-m/(g-k)$, $4(2k-g)+ m\leq 4$ and $0\leq k\leq g/2$, we have $\lambda_f\geq 4+4(k-1)/(g-k)$. 
Fintushel and Stern \cite{FS2} showed that $(\Sigma_k\times S^2)\sharp 4m\overline{\mathbb{CP}^2}$ admits a genus-$(2k+m-1)$ Lefschetz fibration $f_{FS}$ over $S^2$. 
When $m=g-2k+1$, we find $b_2^+=1$ and that $\lambda_{f_{FS}}=4+4(k-1)/(g-k)$. 
This completes the proof.
\end{proof}

%



\begin{thebibliography}{111}

\bibitem{ABKP} J. Amor$\acute{\textnormal{o}}$s, F. Bogomolov, L.  Katzarkov, and T. Pantev; \emph{Symplectic Lefschetz fibrations with arbitrary fundamental groups,} 
J. Differential Geom. \textbf{54} (2000), no. 3, 489--545.

\bibitem{Ar} S. Ju. Arakelov; \emph{Families of algebraic curves with fixed degeneracies,} 
(Russian) Izv. Akad. Nauk SSSR Ser. Mat. \textbf{35} (1971), 1269--1293.

\bibitem{Be} A. Beauville; \emph{L'application canonique pour les surfaces de type general,}
(French) Invent. Math. \textbf{55} (1979), no. 2, 121--140.

\bibitem{BPV} W. Barth, C. Peters and A. Van de Ven; \emph{Compact complex surfaces,} 
Ergebnisse der Mathematik, Spinger-Verlag Berlin, 1984.


\bibitem{Cat} F. Catanese; \emph{Moduli of surfaces of general type,} 
In: Algebraic geometry open problems (Ravello, 1982), 90--112. Lecture Notes in Math., \textbf{997}, Springer, Berlin (1983). 

\bibitem{Ch} Z. Chen; \emph{On the lower bound of the slope of a non-hyperelliptic fibration of genus 4,} 
Intern. J. Math. \textbf{4} (1993), 367--378.

\bibitem{De} M. Dehn; \emph{Die Gruppe der Abbildungsklassen,} 
Acta Math. \textbf{69} (1938), 135--206.

\bibitem{Do} S. K. Donaldson; \emph{Lefschetz pencils on symplectic manifolds,} 
J. Diff. Geom. \textbf{53} (1999), 205--236.

\bibitem{EG} H. Endo and Y. Gurtas; \emph{Lantern relations and rational blowdowns,} 
Proc. Amer. Math. Soc. \textbf{138} (2010), no. 3, 1131--1142.

\bibitem{EMHM} H. Endo, T. E. Mark, and J. Van Horn-Morris; \emph{Monodromy substitutions and rational blowdowns,} 
J. Topol. \textbf{4} (2011), 227--253.

\bibitem{EN} H. Endo and S. Nagami; \emph{Signature of relations in mapping class groups and non-holomorphic Lefschetz fibrations,} 
Trans. Amer. Math. Soc. \textbf{357} (2005), no. 8, 3179--3199.

\bibitem{FS0} R. Fintushel and R. Stern; \emph{Rational blowdowns of smooth 4-manifolds,} 
J. Differential Geom. \textbf{46}, (1997), no. 2, 181--235.

\bibitem{FS1} R. Fintushel and R. Stern; \emph{Constructions of smooth 4-manifolds,} 
Proceedings of the International Congress of Mathematicians, Vol. II (Berlin, 1998), Doc. Math. Extra Vol. II (1998), 443--452.

\bibitem{FS2} R. Fintushel and R. Stern; \emph{Families of simply connected 4-manifolds with the same Seiberg-Witten invariants,} 
Topology \textbf{43},(2004), no.6 1449--1467.

\bibitem{GS} R. Gompf and A. Stipsicz; \emph{4-manifolds and Kirby calculus,} 
Graduate Studies in Mathematics, vol. 20, American Math.~Society, Providence 1999.

\bibitem{Hi} S. Hirose; \emph{Presentations of periodic maps on oriented closed surfaces of genera up to 4,} 
Osaka Journal of Mathematics, \textbf{47}, (2010), no.2, 385--421.



\bibitem{Jo} D. Johnson; \emph{Homeomorphisms of a surface which act trivially on homology,} 
Proc. Amer. Math. Soc. \textbf{75} (1979), 119--125.

\bibitem{Kas} A. Kas; \emph{On the handlebody decomposition associated to a Lefschetz fibration,} 
Pacific J. Math. \textbf{89} (1980), 89--104.

\bibitem{Kon1} K. Konno; \emph{A note on surfaces with pencils of non-hyperelliptic curves of genus 3,} Osaka
J. Math. \textbf{28} (1991), 737--745.

\bibitem{Kon2} K. Konno; \emph{Non-hyperelliptic fibrations of small genus and certain irregular canonical surfaces,}
Ann. Sc. Norm. Pisa Sup. Ser.IV, vol.XX (1993), 575--595.

\bibitem{Kon96} K. Konno; \emph{Even canonical surfaces with small $K^2$. III,} 
Nagoya Math. J. \textbf{143}, (1996), 1--11.

\bibitem{Kor} M. Korkmaz; \emph{Noncomplex smooth 4-manifolds with Lefschetz fibrations,} 
Internat. Math. Res. Not. (2001), No. 3, 115--128.

\bibitem{Li} T.-J. Li; \emph{Symplectic Parshin-Arakelov inequality,} 
Internat. Math. Res. Not. (2000), 941--954.

\bibitem{Man03} M. Manetti; \emph{Surfaces of Albanese general type and the Severi conjecture,} 
Math. Nachr. \textbf{261}/\textbf{262}, (2003), 105--122.

\bibitem{Ma} Y. Matsumoto; \emph{Lefschetz fibrations of genus two --- a topological approach,} 
Topology and Teichm$\ddot{\textnormal{u}}$ller spaces (Katinkulta, 1995), 123--148, World Sci. Publ., River Edge, NJ, 1996.

\bibitem{MS} H. Miyachi and H. Shiga; \emph{Holonomies and the slope inequality of Lefschetz fibrations,} 
Proc. Amer. Math. Soc. \textbf{139} (2011), no. 4, 1299--1307.

\bibitem{OS} B. Ozbagci and A. Stipsicz; \emph{Noncomplex smooth 4-manifolds with genus-2 Lefschetz fibrations,} 
Proc. Amer. Math. Soc. \textbf{128} (2000), 3125--3128.

\bibitem{Oz} B. Ozbagci; \emph{Signatures of Lefschetz fibrations,} 
Pacific J. Math. \textbf{202} (2002), 99--118.


\bibitem{Pa} R. Pardini; \emph{The Severi inequality $K^2\geq 4\chi$ for surfaces of maximal Albanese dimension,} 

\bibitem{Re} M. Reid; \emph{$\pi_1$ for surfaces with small $K^2$,} 
In: Algebraic geometry (Proc. Summer Meeting, Univ. Copenhagen, Copenhagen, 1978), 534--544. 
Lecture Notes in Math., \textbf{732}, Springer, Berlin (1979)

\bibitem{Ro} V. A. Rohlin; \emph{New results in the theory of four-dimensional manifolds,} 
(Russian) Doklady Akad. Nauk SSSR (N.S.) \textbf{84}, (1952), 221--224.

\bibitem{Se} F. Severi; \emph{La serie canonica e la teoria delle serie principali di gruppi di punti sopra una superficie algebrica,} 
Comment. Math. Helv. \textbf{4}, (1932), 268--326.


\bibitem{Sm2} I. Smith; \emph{Lefschetz fibrations and the Hodge bundle,} 
Geom. Topol. \textbf{3} (1999), 211--233.

\bibitem{Sm3} I. Smith; \emph{Lefschetz pencils and divisors in moduli space,} 
Geom. Topol. \textbf{5} (2001), 579--608.

\bibitem{St} A. Stipsicz; \emph{On the number of vanishing cycles in Lefschetz fibrations,} 
Math. Research Letters \textbf{6} (1999), 449--456.

\bibitem{St2} A. Stipsicz; \emph{Singular fibers in Lefschetz fibrations on manifolds with $b^+_2 = 1$,} 
Topology and its Appl. \textbf{117} (2002), 9--21.

\bibitem{Us} M. Usher; \emph{Minimality and symplectic sums,} 
Int. Math. Res. Not. 2006, Art. ID49857, 1--17.

\bibitem{Xi} G. Xiao; \emph{Fibered algebraic surfaces with low slope,} 
Math. Ann. \textbf{276} (1987), 449--466.

\end{thebibliography}
\end{document}